\documentclass[12pt]{amsart}
\usepackage{amsmath,amsthm,amsfonts,amssymb,amscd,euscript}
\usepackage{color}

\newlength{\defbaselineskip}
\theoremstyle{plain}

\theoremstyle{definition}

\theoremstyle{plain}
\newtheorem{thm}{Theorem}

\newtheorem{lem}[thm]{Lemma}
\newtheorem{cor}[thm]{Corollary}

\theoremstyle{definition}
\newtheorem{defn}[thm]{Definition}

\newtheorem{exmp}[thm]{Example}

\newtheorem{rem}[thm]{Remark}
\numberwithin{equation}{section}

\newcommand{\bigboxs}[1]
{ \multiput(#1)(80,0){2}
 {\line(0,80){80}}
\multiput(#1)(0,80){2}
 {\line(80,0){80}}
}
\newcommand{\boxs}[1]
{ \multiput(#1)(20,0){2}
 {\line(0,20){20}}
\multiput(#1)(0,20){2}
 {\line(20,0){20}}
}

\begin{document}

\title{A combinatorial formula for rank 2 cluster variables}
\author{Kyungyong Lee and Ralf Schiffler}
\thanks{Research of K.L. is partially supported by NSF grant DMS 0901367. Research of R.S. is partially supported by NSF grant DMS 1001637. }

\address{Department of Mathematics, Wayne State University, Detroit, MI 48202}
\email{{\tt klee@math.wayne.edu}}
\address{Department of Mathematics, University of Connecticut, Storrs, CT 06269}
\email{{\tt schiffler@math.uconn.edu}}

\begin{abstract} Let $r$ be any positive integer, and let  $x_1, x_2$ be indeterminates. We consider the sequence $\{x_n\}$ defined by the recursive relation $$
x_{n+1} =(x_n^r +1)/{x_{n-1}}
$$
for any integer $n$.
Finding a combinatorial expression for $x_n$ as a rational function of $x_1$ and $x_2$ has been an open problem since 2001.  We give a direct elementary formula for $x_n$ in terms of subpaths of a specific lattice path in the plane. The formula is manifestly positive, providing a new proof of a result by Nakajima and Qin.

\end{abstract}

 \maketitle
 

\textbf{Keywords :} Laurent polynomials, Cluster algebras, Lattice paths.


\section{introduction}
 Let $r$ be any positive integer, and  $x_1, x_2$ be indeterminates. We consider the sequence $\{x_n\}$ defined by the recursive relation $$
x_{n+1} =(x_n^r +1)/{x_{n-1}}
$$
for any integer $n$. These seemingly elementary sequences have turned out to be stubbornly hard to analyze.
 In the case $r=1$, surprisingly, we obtain a periodic sequence of Laurent polynomials of $x_1$ and $x_2$:
$$
x_3=\frac{x_2+1}{x_1}, \,\, x_4=\frac{x_1+x_2+1}{x_1x_2}, \,\, x_5=\frac{x_1+1}{x_2}, \,\, x_6=x_1, \,\, x_7=x_2,...
$$
In all other cases $r>1$, the sequence is without repetition.
It is an important problem in the theory of cluster algebras to understand the sequence for an arbitrary positive integer $r$, because each sequence forms the set of cluster variables of a rank two cluster algebra.

Cluster algebras have been introduced by Fomin and Zelevinsky in \cite{FZ} in the context of total positivity and canonical bases in Lie theory. Since then cluster algebras have been shown to be related to various fields in mathematics including representation theory of finite dimensional algebras, Teichm\"uller theory, Poisson geometry, combinatorics, Lie theory, tropical geometry and mathematical physics.

A cluster algebra is a subalgebra of a field of rational functions in $n$ variables $x_1,x_2,\ldots,x_n$, given by specifying a set of generators, the so-called \emph{cluster variables}. These generators are constructed in a recursive way, starting from the initial variables $x_1,x_2,\ldots,x_n$, by a procedure called \emph{mutation}, which is determined by the choice of a skew symmetric $n\times n$ integer matrix $B$.
Although each mutation is an elementary operation, it is very difficult to compute cluster variables in general, because of the recursive character of the construction.

Finding explicit computable direct formulas for the cluster variables is one of the main open problems in the theory of cluster algebras and has been studied by many mathematicians. Fomin and Zelevinsky showed in \cite{FZ} that every cluster variable is a Laurent polynomial in the initial variables $x_1,x_2,\ldots, x_n$, and they conjectured that this Laurent polynomial has positive coefficients. This \emph{positivity conjecture} has been proved in several special cases; the most general ones being  cluster algebras from surfaces, see \cite{MSW}, cluster algebras that have a bipartite seed, see \cite{N,HL} and cluster algebras whose initial seed is acyclic \cite{Q}.

Direct formulas for the Laurent polynomials have been obtained in several special cases. The most general results are the following:
\begin{itemize}
\item a formula involving the Euler-Poincar\'e characteristic of quiver Grassmannians obtained in \cite{FK} using categorification and generalizing results in \cite{CC,CK1}. While this formula shows a very interesting connection between cluster algebras and geometry, it is of limited computational use, since the Euler-Poincar\'e characteristics of quiver Grassmannians are hard to compute.
\item an elementary combinatorial formula for cluster algebras from surfaces given in \cite{MSW} building on  \cite{S2,ST,S3}.
\item a formula for cluster variables corresponding to string modules as a product of $2\times 2$ matrices obtained in \cite{ADSS}, generalizing a result in \cite{ARS}.
\end{itemize}

In this paper, we consider cluster algebras of rank $n=2$, that is, the integer matrix $B$ is of the form
\[\left[\begin{array}{cc} 0&r\\-r&0\end{array}\right],\] and the cluster variables form the sequence $\{x_n\}$ given by the recursion above.
The rank 2 case is considerably simpler than the general case, but even so, the problem of finding an elementary formula for the cluster variables which also shows positivity had not been solved yet. In \cite{SZ, MP, CZ, DK} such formulas are given in the case $r=2$, which is also a special case of the ones considered in \cite{MSW,ST,ADSS}. In \cite{N,Q} the positivity conjecture was proved for arbitrary $r$ using Euler-Poincar\'e characteristics,
and  in \cite{L} an elementary formula is proved but it does not show positivity. We also want to point out that the positivity conjecture has been shown in the case where the $2\times 2$ matrix $B$ is skew-symmetrizable in \cite{D}, again using Euler-Poincar\'e characteristics.

The main result of this paper is a direct combinatorial formula which computes the cluster variables as a sum of monomials  each with coefficient 1. Clearly, this formula shows positivity. Moreover, each monomial is explicitly given, and the sum is parametrized by subpaths of a specific lattice path in the plane.

\noindent \emph{Acknowledgements.}  We are grateful to Andrei Zelevinsky and Hugh Thomas for valuable suggestions. The first author would like to thank Philippe Di Francesco, Sergey Fomin, Rob Lazarsfeld, and Gregg Musiker for encouraging him to work on this project. The computational part of our research was aided by the commutative algebra package Macaulay 2 \cite{M2}.

\section{Main Result}

Fix a positive integer $r\geq 2$. 

\begin{defn}\label{cn}
Let $\{c_n\}$ be the sequence  defined by the recurrence relation $$c_n=rc_{n-1} -c_{n-2},$$ with the initial condition $c_1=0$, $c_2=1$. When $r=2$, $c_n=n-1$. When $r>2$, it is easy to see that 
$$\aligned
c_n&= \frac{1}{\sqrt{r^2-4}  }\left(\frac{r+\sqrt{r^2-4}}{2}\right)^{n-1} - \frac{1}{\sqrt{r^2-4}  }\left(\frac{r-\sqrt{r^2-4}}{2}\right)^{n-1}\\ &= \sum_{i\geq 0} (-1)^i { {n-2-i} \choose i }r^{n-2-2i}.
\endaligned$$ For example, for $r=3$, the sequence $c_n$ takes the following values:
$$0,1,3,8,21,55,144,...$$\qed
\end{defn}

In order to state our theorem, we fix an integer $n\geq 4$. Consider a rectangle with vertices $(0,0),(0,c_{n-2}),(c_{n-1}-c_{n-2},c_{n-2})$ and $(c_{n-1}-c_{n-2},0)$. In what follows, by the diagonal we mean the line segment from $(0,0)$ to $(c_{n-1}-c_{n-2},c_{n-2})$. A Dyck path is a lattice path
from $(0, 0)$  to $(c_{n-1}-c_{n-2},c_{n-2})$ that proceeds by NORTH or EAST steps and
never goes above the diagonal.

\begin{defn}
A Dyck path below the diagonal is said to be maximal if no subpath of any other Dyck path lies above it. The maximal Dyck path, denoted by $\mathcal{D}_n$, consists of $(w_0, \alpha_1,w_1,\cdots, \alpha_{c_{n-1}}, w_{c_{n-1}})$, where $w_0,\cdots,w_{c_{n-1}}$ are vertices and $\alpha_1,\cdots, \alpha_{c_{n-1}}$ are edges, such that $w_0=(0,0)$ is the south-west corner of the rectangle, $\alpha_i $ connects $w_{i-1}$ and $w_i$, and $w_{c_{n-1}}=(c_{n-1}-c_{n-2},c_{n-2})$ is the north-east corner of the rectangle.
\end{defn}

\begin{rem}
The word obtained from $\mathcal{D}_n$ by forgetting the vertices $w_i$ and replacing each horizontal edge by the letter $x$ and each vertical edge by the letter $y$ is the Christoffel word of slope $c_{n-2}/(c_{n-1}-c_{n-2})$.
\end{rem}

\begin{exmp}Let $r=3$ and $n=5$. Then $\mathcal{D}_5$ is illustrated as follows.
$$
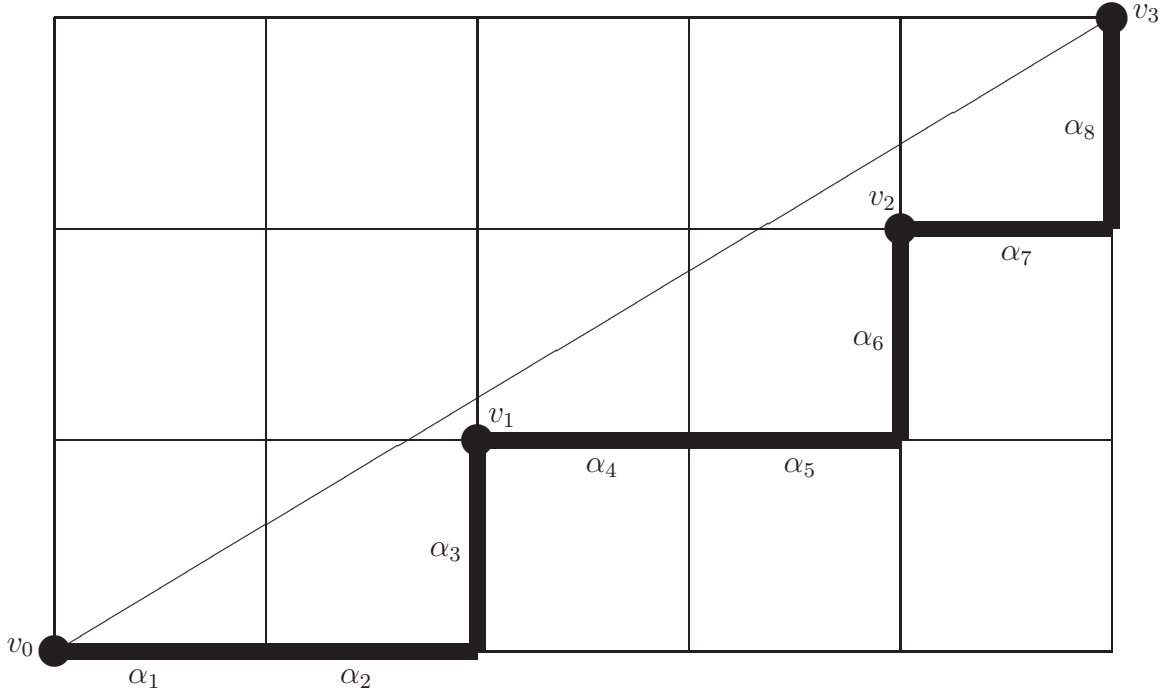
\hspace{16pt} \begin{picture}(400,250)
\bigboxs{0,0}\bigboxs{80,0}\bigboxs{160,0}\bigboxs{240,0}\bigboxs{320,0}
\bigboxs{0,80}\bigboxs{80,80}\bigboxs{160,80}\bigboxs{240,80}\bigboxs{320,80}
\bigboxs{0,160}\bigboxs{80,160}\bigboxs{160,160}\bigboxs{240,160}\bigboxs{320,160}
\put(28,-12){$\tiny{\alpha_1}$}\put(108,-12){$\tiny{\alpha_2}$}
\put(201,68){$\tiny{\alpha_4}$}\put(276,68){$\tiny{\alpha_5}$}
\put(358,148){$\tiny{\alpha_7}$}
\put(142,36){$\tiny{\alpha_3}$}\put(302,116){$\tiny{\alpha_6}$}\put(382,196){$\tiny{\alpha_8}$}
\linethickness{1pt}\put(0,0){\line(5,3){400}}
\linethickness{6pt}\put(0,0){\line(1,0){160}}
\linethickness{6pt}\put(160,0){\line(0,1){80}}
\linethickness{6pt}\put(160,80){\line(1,0){160}}
\linethickness{6pt}\put(320,80){\line(0,1){80}}
\linethickness{6pt}\put(320,160){\line(1,0){80}}
\linethickness{6pt}\put(400,160){\line(0,1){80}}
\put(0,0){\circle*{12}}\put(160,80){\circle*{12}}\put(320,160){\circle*{12}}\put(400,240){\circle*{12}}\put(-18,0){$v_0$}\put(164,87){$v_1$}\put(308,169){$v_2$}\put(408,240){$v_3$}
\end{picture}$$
\end{exmp}

\bigskip
\begin{defn}
Let $i_1<\cdots<i_{c_{n-2}}$ be the sequence of integers such that $\alpha_{i_j}$ is vertical for any $1\leq j\leq c_{n-2}$. Define a sequence $v_0,v_1,\cdots,v_{c_{n-2}}$ of vertices by $v_0=(0,0)$ and $v_j=w_{i_j}$. 
\end{defn}

\begin{defn}
For any $i<j$, let $s_{i,j}$ be the slope of the line through $v_i$ and $v_j$. Let $s$ be the slope of the diagonal, that is, $s=s_{0,c_{n-2}}$. 
\end{defn}

\begin{defn}\label{alpha(i,k)}
For any $0\leq i<k\leq c_{n-2}$, let $\alpha(i,k)$ be the subpath of $\mathcal{D}_n$ defined as follows (for illustrations see Example~\ref{mainexmp}).

\noindent (1) If $s_{i,t}\leq s$ for all $t$ such that $i<t\leq k$, then let $\alpha(i,k)$ be the subpath from $v_i$ to $v_k$. Each of these subpaths will be called a BLUE subpath. See Example~\ref{mainexmp}.

\noindent (2) If $s_{i,t}> s$ for some $i<t\leq k$, then

(2-a) if the smallest such $t$ is of the form $i+c_m-wc_{m-1}$ for some integers $3\le m\le n-2$ and $1\leq w< r-1$, then let $\alpha(i,k)$ be the subpath from $v_i$ to $v_k$. Each of these subpaths will be called a GREEN subpath. When $m$ and $w$ are specified, it will be said to be $(m,w)$-green.

(2-b) otherwise, let $\alpha(i,k)$ be the subpath from the immediate predecessor of $v_i$ to $v_k$. Each of these subpaths will be called a RED subpath.
\qed\end{defn}

Note that every pair $(i,k)$ defines exactly one subpath $\alpha(i,k)$. We denote the set of all these subpaths together with the single edges $\alpha_i$ by $\mathcal{P}(\mathcal{D}_n)$, that is,$$\mathcal{P}(\mathcal{D}_n)=\{\alpha(i,k)\,|\, 0\leq i<k\leq c_{n-2}\} \cup \{\alpha_1,\cdots,\alpha_{c_{n-1}} \}.$$

Now we define a set $\mathcal{F}(\mathcal{D}_n)$ of certain sequences of non-overlapping subpaths of $\mathcal{D}_n$. This set will parametrize the monomials in our expansion formula.

\begin{defn}
Let $$\aligned \mathcal{F}(\mathcal{D}_n)=\{\{\beta_1,\cdots,\beta_t\}\,|&\,t\geq 0,\,\beta_j\in \mathcal{P}(\mathcal{D}_n)\text{ for all }1\leq j\leq t,\,\\ &\text{ if }j\neq j'\text{ then }\beta_j\text{ and }\beta_{j'}\text{ have no common edge,}\\ &\text{ if }\beta_j=\alpha(i,k)\text{ and }\beta_{j'}=\alpha(i',k')\text{ then }i\neq k'\text{ and }i'\neq k,\\
&\text{ and if }\beta_j\text{ is }(m,w)\text{-green then at least one of the }(c_{m-1}-wc_{m-2})\\&\,\,\,\,\,\,\,\,\,\,\,\,\,\,\text{ preceding edges of }v_i\text{ is contained in some }\beta_{j'}\}.  \endaligned$$ \qed
\end{defn}

For any $\mathbf{\beta}=\{\beta_1,\cdots,\beta_t\}$, let $|\mathbf{\beta}|_2$ be the total number of edges in $\beta_1,\cdots,\beta_t$, and $|\mathbf{\beta}|_1=\sum_{j=1}^t |\beta_j|_1$, where
$$
|\beta_j|_1=\left\{\begin{array}{ll} 0, &\text{ if }\beta_j=\alpha_i\text{ for some }1\leq i\leq c_{n-1} \\k-i, &\text{ if }\beta_j=\alpha(i,k)\text{ for some }0\leq i<k\leq c_{n-2}.  \end{array}   \right.
$$

\begin{thm}\label{mainthm}
Let $r\geq 2$ be a positive integer. Let $x_1, x_2$ be indeterminates. Define the sequence $\{x_n\}$ by $$
x_{n+1} =(x_n^r +1)/{x_{n-1}} \text{ for any integer }n.
$$ 
Then for $n\geq 4$,
\begin{equation}\label{maineq1}x_n=x_1^{-c_{n-1}} x_2^{-c_{n-2}}\sum_{\mathbf{\beta}\in\mathcal{F}(\mathcal{D}_n)}x_1^{r|\mathbf{\beta}|_1}x_2^{r(c_{n-1}-|\mathbf{\beta}|_2)}\end{equation}and
\begin{equation}\label{maineq2}x_{3-n}=x_2^{-c_{n-1}} x_1^{-c_{n-2}}\sum_{\mathbf{\beta}\in\mathcal{F}(\mathcal{D}_n)}x_2^{r|\mathbf{\beta}|_1}x_1^{r(c_{n-1}-|\mathbf{\beta}|_2)}.\end{equation}
\end{thm}

Here $x_n$ $(n\in\mathbb{Z})$ are called cluster variables, and the cluster algebra of rank 2 is the $\mathbb{Q}$-subalgebra generated by all cluster variables $x_n$ in the field of rational functions in the commutative variables $x_1$ and $x_2$.

\begin{exmp}\label{mainexmp} Let $r=3$ and $n=5$. If the edge $\alpha_i$ is marked $\begin{picture}(25,10)\linethickness{3pt}\put(2,2){\line(1,0){20}}\color{white}\put(9,2){\line(1,0){6}} \end{picture}$, then $\alpha_i$ can occur in $\beta$.
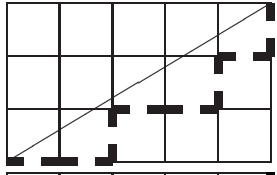
$$
\begin{array}{cc}
\hspace{16pt} \begin{picture}(100,60)
\boxs{0,0}\boxs{20,0}\boxs{40,0}\boxs{60,0}\boxs{80,0}
\boxs{0,20}\boxs{20,20}\boxs{40,20}\boxs{60,20}\boxs{80,20}
\boxs{0,40}\boxs{20,40}\boxs{40,40}\boxs{60,40}\boxs{80,40}
\linethickness{1pt}\put(0,0){\line(5,3){100}}
\linethickness{3pt}\put(0,0){\line(1,0){40}}
\linethickness{3pt}\put(40,0){\line(0,1){20}}
\linethickness{3pt}\put(40,20){\line(1,0){40}}
\linethickness{3pt}\put(80,20){\line(0,1){20}}
\linethickness{3pt}\put(80,40){\line(1,0){20}}
\linethickness{3pt}\put(100,40){\line(0,1){20}}\color{white}\put(47,20){\line(1,0){6}}\linethickness{3pt}\color{white}\put(67,20){\line(1,0){6}}\color{white}\put(7,0){\line(1,0){6}}\linethickness{3pt}\color{white}\put(27,0){\line(1,0){6}}\color{white}\put(87,40){\line(1,0){6}}\put(40,7){\line(0,1){6}}\put(80,27){\line(0,1){6}}\put(100,47){\line(0,1){6}}
\end{picture}
&\,\,\,\,\,\,\,\,\,\,\,\,\,\,\,\,\,\,\,\,\,\,\,\,\,\,\,\,\,\,\,\,\,\,\,\,\,\,\,\,\,\,\,\,\,\,\,\,\,\,\,\,\,\begin{picture}(300,60)\put(-80,35){$\tiny{\sum_{\begin{array}{l}\mathbf{\beta}\subset \{\alpha_1,\cdots,\alpha_{8} \}  \end{array}}x_1^{r|\mathbf{\beta}|_1}x_2^{r(c_{n-1}-|\mathbf{\beta}|_2)}}$}\put(-80,15){$=(1+x_2^3)^8$}\end{picture}
\\
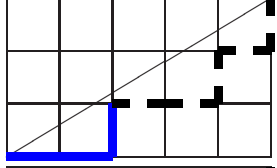
\hspace{16pt} \begin{picture}(100,60)
\boxs{0,0}\boxs{20,0}\boxs{40,0}\boxs{60,0}\boxs{80,0}
\boxs{0,20}\boxs{20,20}\boxs{40,20}\boxs{60,20}\boxs{80,20}
\boxs{0,40}\boxs{20,40}\boxs{40,40}\boxs{60,40}\boxs{80,40}
\linethickness{1pt}\put(0,0){\line(5,3){100}}
\linethickness{3pt}\color{blue}\put(0,0){\line(1,0){40}}
\linethickness{3pt}\color{blue}\put(40,0){\line(0,1){20}}
\linethickness{3pt}\color{black}\put(40,20){\line(1,0){40}}
\linethickness{3pt}\put(80,20){\line(0,1){20}}
\linethickness{3pt}\put(80,40){\line(1,0){20}}
\linethickness{3pt}\put(100,40){\line(0,1){20}}\color{white}\put(47,20){\line(1,0){6}}\color{white}\put(67,20){\line(1,0){6}}\put(87,40){\line(1,0){6}}\put(80,27){\line(0,1){6}}\put(100,47){\line(0,1){6}}
\end{picture}
&\,\,\,\,\,\,\,\,\,\,\,\,\,\,\,\,\,\,\,\,\,\,\,\,\,\,\,\,\,\,\,\,\,\,\,\,\,\,\,\,\,\,\,\,\,\,\,\,\,\,\,\,\,\begin{picture}(300,60)\put(-80,35){$\tiny{\sum_{\begin{array}{l}\{\alpha(0,1)\}\subset\mathbf{\beta}\subset\{\alpha(0,1)\}\cup  \{\alpha_4,\cdots,\alpha_{8} \}  \end{array}}x_1^{r|\mathbf{\beta}|_1}x_2^{r(c_{n-1}-|\mathbf{\beta}|_2)}}$}\put(-80,15){$= x_1^3(1+x_2^3)^5$}\end{picture}\\
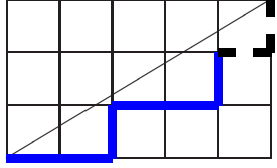
\hspace{16pt} \begin{picture}(100,60)
\boxs{0,0}\boxs{20,0}\boxs{40,0}\boxs{60,0}\boxs{80,0}
\boxs{0,20}\boxs{20,20}\boxs{40,20}\boxs{60,20}\boxs{80,20}
\boxs{0,40}\boxs{20,40}\boxs{40,40}\boxs{60,40}\boxs{80,40}
\linethickness{1pt}\put(0,0){\line(5,3){100}}
\linethickness{3pt}\color{blue}\put(0,0){\line(1,0){40}}
\linethickness{3pt}\color{blue}\put(40,0){\line(0,1){20}}
\linethickness{3pt}\put(40,20){\line(1,0){40}}
\linethickness{3pt}\put(80,20){\line(0,1){20}}
\linethickness{3pt}\color{black}\put(80,40){\line(1,0){20}}
\linethickness{3pt}\put(100,40){\line(0,1){20}}\color{white}\put(87,40){\line(1,0){6}}\put(100,47){\line(0,1){6}}
\end{picture}
&\,\,\,\,\,\,\,\,\,\,\,\,\,\,\,\,\,\,\,\,\,\,\,\,\,\,\,\,\,\,\,\,\,\,\,\,\,\,\,\,\,\,\,\,\,\,\,\,\,\,\,\,\,\begin{picture}(300,60)\put(-80,35){$\tiny{\sum_{\begin{array}{l}\{\alpha(0,2)\}\subset\mathbf{\beta}\subset\{\alpha(0,2)\}\cup  \{\alpha_7,\alpha_{8} \}  \end{array}}x_1^{r|\mathbf{\beta}|_1}x_2^{r(c_{n-1}-|\mathbf{\beta}|_2)}}$}\put(-80,15){$= x_1^6(1+x_2^3)^2$}\end{picture}\\
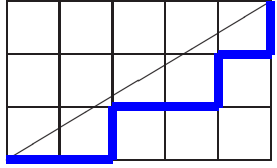
\hspace{16pt} \begin{picture}(100,60)
\boxs{0,0}\boxs{20,0}\boxs{40,0}\boxs{60,0}\boxs{80,0}
\boxs{0,20}\boxs{20,20}\boxs{40,20}\boxs{60,20}\boxs{80,20}
\boxs{0,40}\boxs{20,40}\boxs{40,40}\boxs{60,40}\boxs{80,40}
\linethickness{1pt}\put(0,0){\line(5,3){100}}
\linethickness{3pt}\color{blue}\put(0,0){\line(1,0){40}}
\linethickness{3pt}\color{blue}\put(40,0){\line(0,1){20}}
\linethickness{3pt}\put(40,20){\line(1,0){40}}
\linethickness{3pt}\put(80,20){\line(0,1){20}}
\linethickness{3pt}\put(80,40){\line(1,0){20}}
\linethickness{3pt}\put(100,40){\line(0,1){20}}
\end{picture}
&\,\,\,\,\,\,\,\,\,\,\,\,\,\,\,\,\,\,\,\,\,\,\,\,\,\,\,\,\,\,\,\,\,\,\,\,\,\,\,\,\,\,\,\,\,\,\,\,\,\,\,\,\,\begin{picture}(300,60)\put(-80,35){$\tiny{\sum_{\begin{array}{l}\mathbf{\beta}=\{\alpha(0,3)\} \end{array}}x_1^{r|\mathbf{\beta}|_1}x_2^{r(c_{n-1}-|\mathbf{\beta}|_2)}}$}\put(-80,15){$= x_1^9$}\end{picture}\\
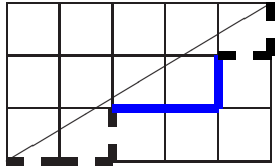
\hspace{16pt} \begin{picture}(100,60)
\boxs{0,0}\boxs{20,0}\boxs{40,0}\boxs{60,0}\boxs{80,0}
\boxs{0,20}\boxs{20,20}\boxs{40,20}\boxs{60,20}\boxs{80,20}
\boxs{0,40}\boxs{20,40}\boxs{40,40}\boxs{60,40}\boxs{80,40}
\linethickness{1pt}\put(0,0){\line(5,3){100}}
\linethickness{3pt}\put(0,0){\line(1,0){40}}
\linethickness{3pt}\put(40,0){\line(0,1){20}}
\linethickness{3pt}\color{blue}\put(40,20){\line(1,0){40}}
\linethickness{3pt}\put(80,20){\line(0,1){20}}
\linethickness{3pt}\color{black}\put(80,40){\line(1,0){20}}
\linethickness{3pt}\put(100,40){\line(0,1){20}}\color{white}\put(87,40){\line(1,0){6}}\put(7,0){\line(1,0){6}}\put(27,0){\line(1,0){6}}\put(40,7){\line(0,1){6}}\put(100,47){\line(0,1){6}}
\end{picture}
&\,\,\,\,\,\,\,\,\,\,\,\,\,\,\,\,\,\,\,\,\,\,\,\,\,\,\,\,\,\,\,\,\,\,\,\,\,\,\,\,\,\,\,\,\,\,\,\,\,\,\,\,\,\begin{picture}(300,60)\put(-80,35){$\tiny{\sum_{\begin{array}{l}\{\alpha(1,2)\}\subset \mathbf{\beta}\subset\{\alpha(1,2)\}\cup  \{\alpha_1,\alpha_2,\alpha_{3},\alpha_{7},\alpha_{8} \}  \end{array}}x_1^{r|\mathbf{\beta}|_1}x_2^{r(c_{n-1}-|\mathbf{\beta}|_2)}}$}\put(-80,15){$= x_1^3(1+x_2^3)^5$}\end{picture}\\
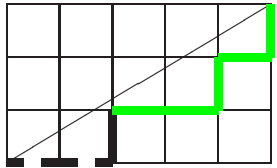
\hspace{16pt} \begin{picture}(100,60)
\boxs{0,0}\boxs{20,0}\boxs{40,0}\boxs{60,0}\boxs{80,0}
\boxs{0,20}\boxs{20,20}\boxs{40,20}\boxs{60,20}\boxs{80,20}
\boxs{0,40}\boxs{20,40}\boxs{40,40}\boxs{60,40}\boxs{80,40}
\linethickness{1pt}\put(0,0){\line(5,3){100}}
\linethickness{3pt}\put(0,0){\line(1,0){40}}
\linethickness{3pt}\put(40,0){\line(0,1){20}}
\linethickness{3pt}\color{green}\put(40,20){\line(1,0){40}}
\linethickness{3pt}\put(80,20){\line(0,1){20}}
\linethickness{3pt}\put(80,40){\line(1,0){20}}
\linethickness{3pt}\put(100,40){\line(0,1){20}}\color{white}\put(7,0){\line(1,0){6}}\put(27,0){\line(1,0){6}}
\end{picture}
&\,\,\,\,\,\,\,\,\,\,\,\,\,\,\,\,\,\,\,\,\,\,\,\,\,\,\,\,\,\,\,\,\,\,\,\,\,\,\,\,\,\,\,\,\,\,\,\,\,\,\,\,\,\begin{picture}(300,60)\put(-80,35){$\tiny{\sum_{\begin{array}{l}\{\alpha(1,3)\}\cup  \{\alpha_3\}\subset\mathbf{\beta}\subset\{\alpha(1,3)\}\cup  \{\alpha_1,\alpha_2,\alpha_3 \}  \end{array}}x_1^{r|\mathbf{\beta}|_1}x_2^{r(c_{n-1}-|\mathbf{\beta}|_2)}}$}\put(-80,15){$= x_1^6(1+x_2^3)^2$}\end{picture}\\
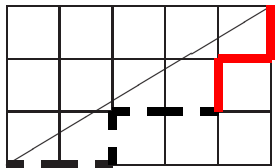
\hspace{16pt} \begin{picture}(100,60)
\boxs{0,0}\boxs{20,0}\boxs{40,0}\boxs{60,0}\boxs{80,0}
\boxs{0,20}\boxs{20,20}\boxs{40,20}\boxs{60,20}\boxs{80,20}
\boxs{0,40}\boxs{20,40}\boxs{40,40}\boxs{60,40}\boxs{80,40}
\linethickness{1pt}\put(0,0){\line(5,3){100}}\linethickness{3pt}
\linethickness{3pt}\put(0,0){\line(1,0){40}}\color{white}\put(7,0){\line(1,0){6}}\linethickness{3pt}\color{white}\put(27,0){\line(1,0){6}}
\linethickness{3pt}\color{black}\put(40,0){\line(0,1){20}}
\linethickness{3pt}\put(40,20){\line(1,0){40}}\color{white}\put(47,20){\line(1,0){6}}\linethickness{3pt}\color{white}\put(67,20){\line(1,0){6}}
\linethickness{3pt}\color{red}\put(80,20){\line(0,1){20}}
\linethickness{3pt}\put(80,40){\line(1,0){20}}
\linethickness{3pt}\put(100,40){\line(0,1){20}}\color{white}\put(40,7){\line(0,1){6}}
\end{picture}
&\,\,\,\,\,\,\,\,\,\,\,\,\,\,\,\,\,\,\,\,\,\,\,\,\,\,\,\,\,\,\,\,\,\,\,\,\,\,\,\,\,\,\,\,\,\,\,\,\,\,\,\,\,\begin{picture}(300,60)\put(-80,35){$\tiny{\sum_{\begin{array}{l} \{\alpha(2,3)\}\subset\mathbf{\beta}\subset\{\alpha(2,3)\}\cup  \{\alpha_1,\cdots,\alpha_{5} \}  \end{array}}x_1^{r|\mathbf{\beta}|_1}x_2^{r(c_{n-1}-|\mathbf{\beta}|_2)}}$}\put(-80,15){$= x_1^3(1+x_2^3)^5$}\end{picture}\\
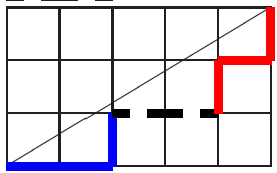
\hspace{16pt} \begin{picture}(100,60)
\boxs{0,0}\boxs{20,0}\boxs{40,0}\boxs{60,0}\boxs{80,0}
\boxs{0,20}\boxs{20,20}\boxs{40,20}\boxs{60,20}\boxs{80,20}
\boxs{0,40}\boxs{20,40}\boxs{40,40}\boxs{60,40}\boxs{80,40}
\linethickness{1pt}\put(0,0){\line(5,3){100}}
\linethickness{3pt}\color{blue}\put(0,0){\line(1,0){40}}
\linethickness{3pt}\put(40,0){\line(0,1){20}}
\linethickness{3pt}\color{black}\put(40,20){\line(1,0){40}}
\linethickness{3pt}\color{white}\put(47,20){\line(1,0){6}}\linethickness{3pt}\color{white}\put(67,20){\line(1,0){6}}
\linethickness{3pt}\color{red}\put(80,20){\line(0,1){20}}
\linethickness{3pt}\put(80,40){\line(1,0){20}}
\linethickness{3pt}\put(100,40){\line(0,1){20}}
\end{picture}
&\,\,\,\,\,\,\,\,\,\,\,\,\,\,\,\,\,\,\,\,\,\,\,\,\,\,\,\,\,\,\,\,\,\,\,\,\,\,\,\,\,\,\,\,\,\,\,\,\,\,\,\,\,\begin{picture}(300,60)\put(-80,35){$\tiny{\sum_{\begin{array}{l} \{\alpha(0,1),\alpha(2,3)\}\subset\mathbf{\beta}\subset\{\alpha(0,1),\alpha(2,3)\}\cup  \{\alpha_4,\alpha_{5} \}  \end{array}}x_1^{r|\mathbf{\beta}|_1}x_2^{r(c_{n-1}-|\mathbf{\beta}|_2)}}$}\put(-80,15){$= x_1^6(1+x_2^3)^2$}\end{picture}\end{array}
$$Adding the above 8 polynomials together gives 
\begin{equation}\label{8eqtoge}\aligned  &{x_2}^{24}+8 {x_2}^{21}+3 {x_1}^{3} {x_2}^{15}+28
      {x_2}^{18}+15 {x_1}^{3} {x_2}^{12}+56 {x_2}^{15}+3
      {x_1}^{6} {x_2}^{6}+30 {x_1}^{3} {x_2}^{9}+70
      {x_2}^{12}\\&+{x_1}^{9}+6 {x_1}^{6} {x_2}^{3}+30
      {x_1}^{3} {x_2}^{6}+56 {x_2}^{9}+3 {x_1}^{6}+15
      {x_1}^{3} {x_2}^{3}+28 {x_2}^{6}+3 {x_1}^{3}+8
      {x_2}^{3}+1.\endaligned\end{equation}
      Then $x_5$ is obtained by dividing (\ref{8eqtoge}) by $x_1^8 x_2^3$.

\end{exmp}

We also obtain the following formula for the $F$-polynomials. Let $g_\ell$ be the $g$-vector and let $F_\ell$ be the $F$-polynomial of $x_\ell$, for all integers $\ell$. 
{Then $g_3=(-1,r), \,g_0=(0,-1),\,F_3=y_1+1$ and $F_0=y_2+1$, and all other cases are described in the following result.}
\begin{cor} Let $n\ge 4$. Then
\[ \begin{array}{lrclcccrcl}
&g_n&=&(-c_{n-1},c_n)& \quad &,&\quad& g_{3-n}&=&(-c_{n-2},c_{n-3}) ,\\
and \quad \\
& F_n&=& \sum_{\mathbf{\beta}\in\mathcal{F}(\mathcal{D}_n)}y_1^{|\mathbf{\beta}|_2}y_2^{|\mathbf{\beta}|_1}&&,& &
F_{3-n}&=&\sum_{\mathbf{\beta}\in\mathcal{F}(\mathcal{D}_n)}y_1^{c_{n-2}-|\mathbf{\beta}|_1}y_2^{c_{n-1}-|\mathbf{\beta}|_2}.
\end{array}\]
\end{cor}
\begin{proof}
The formulas for the $g$-vectors follow easily from the recursive relations for $g$-vectors given in \cite[Proposition 6.6]{FZ4}. To prove the formulas for the $F$-polynomials, we work
in the cluster algebra with principal coefficients at the seed \[\Sigma=\left((x_1,x_2),(y_1,y_2),\left[\begin{array}{cc} 0&r\\-r&0\\1&0\\0&1\end{array}\right] \right).\]
In this cluster algebra, the Laurent expansion in $\Sigma$ of any cluster variable $x_\ell$ is  homogeneous with respect to the $\mathbb{Z}^2$-grading 
\[ \textup{deg}\, x_1 = (1,0), \ \textup{deg} \,x_2 =(0,1)\  ,\textup{deg}\, y_1 = (0,r), \ \textup{deg}\, y_2 =(-r,0),\]
and, moreover,   deg\,$x_\ell=g_\ell$, see \cite[Proposition 6.1]{FZ4}.
It follows that the expansion formulas with principal coefficients are of the form

\begin{equation}\label{principal1}x_n=x_1^{-c_{n-1}} x_2^{-c_{n-2}}\sum_{\mathbf{\beta}\in\mathcal{F}(\mathcal{D}_n)}x_1^{r|\mathbf{\beta}|_1}x_2^{r(c_{n-1}-|\mathbf{\beta}|_2)}y_1^{a_1}y_2^{a_2}\end{equation}and
\begin{equation}\label{principal2}x_{3-n}=x_2^{-c_{n-1}} x_1^{-c_{n-2}}\sum_{\mathbf{\beta}\in\mathcal{F}(\mathcal{D}_n)}x_2^{r|\mathbf{\beta}|_1}x_1^{r(c_{n-1}-|\mathbf{\beta}|_2)}y_1^{b_1}y_2^{b_2},\end{equation}
where $a_1,a_2,b_1 $ and $b_2$ are integers such that 
\[g_n=\textup{deg}\, x_n = (-c_{n-1}+ {r|\mathbf{\beta}|_1} -r\,a_2 , {-c_{n-2}}+ {r(c_{n-1}-|\mathbf{\beta}|_2)} +r\, a_1 )\]
and 
\[ g_{3-n}=\textup{deg}\,x_{3-n} =({-c_{n-2}}+{r(c_{n-1}-|\mathbf{\beta}|_2)} -r\,b_2 , 
{-c_{n-1}}+ {r|\mathbf{\beta}|_1} +r\,b_1).\]
Now, using the formulas for the $g$-vectors in the corollary and the recurrence $c_n=r\,c_{n-1}-c_{n-2}$, we get
$a_1=|\mathbf{\beta}|_2$, $a_2=|\mathbf{\beta}|_1$, $b_1=c_{n-2}-|\mathbf{\beta}|_1$ and
$b_2=c_{n-1}-|\mathbf{\beta}|_2$. The formulas for the $F$-polynomials now follow by setting $x_1=x_2=1$ in the equations (\ref{principal1}) and (\ref{principal2}).
 \end{proof}

Theorem~\ref{mainthm} also enables us to compute the Euler-Poincar\'{e} characteristics of  certain quiver Grassmannians. Let $Q_r$ be the generalized Kronecker  quiver with  two vertices 1 and 2, and $r$ arrows from 1 to 2. For $n\geq 3$, let $M(n)$ (resp. $M(3-n)$) be the unique (up to an isomorphism) indecomposable representation of dimension vector $(c_{n-1}, c_{n-2})$ (resp. $(c_{n-2}, c_{n-1})$). Then the indecomposable projective representations are $P(2)=M(0)$ and $P(1)=M(-1)$, and any indecomposable preprojective representation is of the form  $M(3-n)$. Similarly, the indecomposable injective representations are $I(1)=M(3)$ and $I(2)=M(4)$, and any indecomposable preinjective representation is of the form  $M(n)$.  Let $\text{Gr}_{(e_1,e_2)}(M(n))$ (resp. $\text{Gr}_{(e_1,e_2)}(M(3-n))$) be the variety parametrizing all subrepresentations of $M(n)$ (resp. $M(3-n)$) of dimension vector $(e_1,e_2)$. We use a result of Caldero and  Zelevinsky \cite[Theorem 3.2 and (3.5)]{CZ}.

\begin{thm}[Caldero and  Zelevinsky]
Let $n\geq 3$. Then the cluster variable $x_n$ is equal to 
$$
x_1^{-c_{n-1}} x_2^{-c_{n-2}}\sum_{e_1,e_2} \chi(\emph{Gr}_{(e_1,e_2)}(M(n)))x_1^{r(c_{n-2}-e_{2})} x_2^{re_{1}},
$$and $x_{3-n}$ is equal to 
$$
x_1^{-c_{n-2}} x_2^{-c_{n-1}} \sum_{e_1,e_2} \chi(\emph{Gr}_{(e_1,e_2)}(M(3-n)))x_1^{r(c_{n-1}-e_{2})}x_2^{re_1}.
$$
\end{thm}

In the case $n=3$, the representations $M(3)$ and $M(0)$ are simple representation, meaning that they do not have non-trivial subrepresentations. Therefore the Euler-Poincar\'e characteristics of their Grassmannians are non-zero if and only  if the dimension vector $(e_1,e_2)$ is equal to  $(0,0)$ or equal to the dimension of the simple representation itself, in which cases the Euler-Poincar\'e characteristic is 1. For $n\ge 4$, we get the following result.
\begin{cor} 
Let $n\geq 4$. Then for any integers $e_1$ and $e_2$,
$$\chi(\emph{Gr}_{(e_1,e_2)}(M(n)))=\#\{\mathbf{\beta}\in\mathcal{F}(\mathcal{D}_n): \,|\mathbf{\beta}|_1=c_{n-2}-e_2, |\mathbf{\beta}|_2=c_{n-1}-e_1 \}$$
and
$$\chi(\emph{Gr}_{(e_1,e_2)}(M(3-n)))=\#\{\mathbf{\beta}\in\mathcal{F}(\mathcal{D}_n):\,|\mathbf{\beta}|_1=e_1, |\mathbf{\beta}|_2=e_2 \}.$$
\end{cor}

\section{Proofs}

Since (\ref{maineq2}) can be easily obtained from (\ref{maineq1}) by interchanging $x_1$ and $x_2$, we will prove (\ref{maineq1}). In addition to our theoretical proof, our formula is checked by Macaulay 2 for any $r,n$ with $r+n\leq 11$. 
 We need more notation.

\begin{defn}
For integers $u,n$ with $3\leq u\leq n-1$, let $$\aligned\mathcal{T}^{\geq u}(\mathcal{D}_n):=\{\{\beta_1,\cdots,\beta_t\}\,|&\,t\geq 1,\,\beta_j\in \mathcal{P}(\mathcal{D}_n)\text{ for all }1\leq j\leq t,\,\\
 &\text{ if }j\neq j'\text{ then }\beta_j\text{ and }\beta_{j'}\text{ have no common edge,}\\ 
 &\text{ if }\beta_j=\alpha(i,k)\text{ and }\beta_{j'}=\alpha(i',k')\text{ then }i\neq k'\text{ and }i'\neq k,\\
&\text{ and there exist integers }j,w,m,\text{ with }m\geq u\text{ such that }\\
&\,\,\,\,\,\,\,\,\,\,\,\,\,\,\beta_j\text{ is }(m,w)\text{-green and none of the }(c_{m-1}-wc_{m-2})\\&\,\,\,\,\,\,\,\,\,\,\,\,\,\,\text{preceding edges of }v_i\text{ is contained in any }\beta_{j'}\}.
\endaligned$$.
\end{defn}
\begin{defn}
Let $$\aligned \widetilde{\mathcal{F}}(\mathcal{D}_n)=\{\{\beta_1,\cdots,\beta_t\}\,|&\,t\geq 0,\,\beta_j\in \mathcal{P}(\mathcal{D}_n)\text{ for all }1\leq j\leq t,\,\\ &\text{ if }j\neq j'\text{ then }\beta_j\text{ and }\beta_{j'}\text{ have no common edge,}\\ &\text{ and if }\beta_j=\alpha(i,k)\text{ and }\beta_{j'}=\alpha(i',k')\text{ then }i\neq k'\text{ and }i'\neq k\}.  \endaligned$$  \qed
\end{defn}

\begin{lem}\label{20110516eq2}
If $m\geq n-1$, then there do not exist $i,w$  $(1\leq w< r-1)$ such that $\min\{t\,|\,i<t\leq c_{n-2},\,s_{i,t}> s\}$ is of the form $i+c_m-wc_{m-1}$. In particular, for any $n\geq 4$, the set $\mathcal{T}^{\geq n-1}(\mathcal{D}_{n})$ is empty and 
\begin{equation}\label{20110516eq}{\mathcal{F}}(\mathcal{D}_n)=\widetilde{\mathcal{F}}(\mathcal{D}_n)\setminus\mathcal{T}^{\geq 3}(\mathcal{D}_{n}).\end{equation}
 \end{lem}
\begin{proof}
If $m\geq n-1$ and $\min\{t\,|\,i<t\leq c_{n-2},\,s_{i,t}> s\}=i+c_m-wc_{m-1}$, then $\min\{t\,|\,i<t\leq c_{n-2},\,s_{i,t}> s\}\geq c_{n-1}-wc_{n-2}$, which would be greater than $c_{n-2}$ because $w\leq r-2$. But this is a contradiction, because $v_{c_{n-2}}$ is the highest vertex in $\mathcal{D}_{n}$.
\end{proof}

Let $z_3=x_3$ and
\begin{equation}\label{20110411z}z_n=x_1^{-c_{n-1}} x_2^{-c_{n-2}}\sum_{\mathbf{\beta}\in\widetilde{\mathcal{F}}(\mathcal{D}_n)}x_1^{r|\mathbf{\beta}|_1}x_2^{r(c_{n-1}-|\mathbf{\beta}|_2)}\end{equation}for $n\geq 4$.

Let $K = \mathbb{Q}(x_1,x_2)$ be the field of rational functions in the commutative
variables $x_1$ and $x_2$. Let $F$ be the automorphism of $K$, which is defined by
\begin{equation}\label{Kont_map}F : \left\{\begin{array}{l}x_1 \mapsto x_2 \\ x_2 \mapsto \frac{1+x_2^r}{x_1}.\end{array}\right.\end{equation}

\begin{lem}\label{20110411lem1} Let $n\geq 3$. Then
$$z_{n+1}=F(z_{n})+x_1^{-c_{n}} x_2^{-c_{n-1}}\sum_{\mathbf{\beta}\in\mathcal{T}^{\geq 3}(\mathcal{D}_{n+1})\setminus \mathcal{T}^{\geq 4}(\mathcal{D}_{n+1})}x_1^{r|\mathbf{\beta}|_1}x_2^{r(c_{n}-|\mathbf{\beta}|_2)}.$$
\end{lem}

\begin{lem}\label{20110411lem2} Let $u\geq 3$ and $n\geq u+2$. Then
$$\aligned 
& F\left(x_1^{-c_{n-1}} x_2^{-c_{n-2}}\sum_{\mathbf{\beta}\in\mathcal{T}^{\geq u}(\mathcal{D}_n)\setminus \mathcal{T}^{\geq u+1}(\mathcal{D}_n)}x_1^{r|\mathbf{\beta}|_1}x_2^{r(c_{n-1}-|\mathbf{\beta}|_2)}\right)\\
&=x_1^{-c_{n}} x_2^{-c_{n-1}}\sum_{\mathbf{\beta}\in\mathcal{T}^{\geq u+1}(\mathcal{D}_{n+1})\setminus \mathcal{T}^{\geq u+2}(\mathcal{D}_{n+1})}x_1^{r|\mathbf{\beta}|_1}x_2^{r(c_{n}-|\mathbf{\beta}|_2)}.\endaligned$$
\end{lem}

\begin{lem}\label{20110411lem3} Let $n\geq 4$.  Then
\begin{equation}\label{20110411lem3f}\aligned  x_n&=z_n-\sum_{m=5}^n F^{n-m}\left(x_1^{-c_{m-1}} x_2^{-c_{m-2}}\sum_{\mathbf{\beta}\in\mathcal{T}^{\geq 3}(\mathcal{D}_m)\setminus \mathcal{T}^{\geq 4}(\mathcal{D}_m)}x_1^{r|\mathbf{\beta}|_1}x_2^{r(c_{m-1}-|\mathbf{\beta}|_2)}\right)\\
&=x_1^{-c_{n-1}} x_2^{-c_{n-2}}\sum_{\mathbf{\beta}\in\mathcal{F}(\mathcal{D}_n)}x_1^{r|\mathbf{\beta}|_1}x_2^{r(c_{n-1}-|\mathbf{\beta}|_2)}.
\endaligned\end{equation}
\end{lem}

The proof of Lemma~\ref{20110411lem1} will be independent of those of Lemmas~\ref{20110411lem2} and \ref{20110411lem3}. We prove Lemmas~\ref{20110411lem2} and \ref{20110411lem3} by the following induction: 
\begin{equation}\label{globalinduction}
\aligned&[\text{Lemma~\ref{20110411lem2} holds true for }n\leq d] \Longrightarrow \text{[Lemma~\ref{20110411lem3} holds true for }n\leq d+1]\\
& \Longrightarrow [\text{Lemma~\ref{20110411lem2} holds true for }n\leq d+1] \Longrightarrow [\text{Lemma~\ref{20110411lem3} holds true for }n\leq d+2] \cdots.
\endaligned\end{equation}

\begin{proof}[Proof of Lemma~\ref{20110411lem3}]
We use induction on $n$. It is easy to show that $x_4=z_4$. Assume that (\ref{20110411lem3f}) holds for $n$.

Then $$\aligned 
x_{n+1}& =F(x_n)\\
&\overset{F : homomorphism}= F(z_n)-\sum_{m=5}^n F^{n-m+1}\left(x_1^{-c_{m-1}} x_2^{-c_{m-2}}\sum_{\mathbf{\beta}\in\mathcal{T}^{\geq 3}(\mathcal{D}_m)\setminus \mathcal{T}^{\geq 4}(\mathcal{D}_m)}x_1^{r|\mathbf{\beta}|_1}x_2^{r(c_{m-1}-|\mathbf{\beta}|_2)}\right)\\
&\overset{Lemma~\ref{20110411lem1}}= z_{n+1}-\sum_{m=5}^{n+1} F^{n-m+1}\left(x_1^{-c_{m-1}} x_2^{-c_{m-2}}\sum_{\mathbf{\beta}\in\mathcal{T}^{\geq 3}(\mathcal{D}_m)\setminus \mathcal{T}^{\geq 4}(\mathcal{D}_m)}x_1^{r|\mathbf{\beta}|_1}x_2^{r(c_{m-1}-|\mathbf{\beta}|_2)}\right)\\
&\overset{Lemma~\ref{20110411lem2}}= z_{n+1}-\sum_{m=5}^{n+1} x_1^{-c_{n}} x_2^{-c_{n-1}}\sum_{\mathbf{\beta}\in\mathcal{T}^{\geq n-m+4}(\mathcal{D}_{n+1})\setminus \mathcal{T}^{\geq n-m+5}(\mathcal{D}_{n+1})}x_1^{r|\mathbf{\beta}|_1}x_2^{r(c_{n}-|\mathbf{\beta}|_2)}\\
&= z_{n+1}-x_1^{-c_{n}} x_2^{-c_{n-1}}\sum_{\mathbf{\beta}\in\mathcal{T}^{\geq 3}(\mathcal{D}_{n+1})\setminus \mathcal{T}^{\geq n}(\mathcal{D}_{n+1})}x_1^{r|\mathbf{\beta}|_1}x_2^{r(c_{n}-|\mathbf{\beta}|_2)}\\
&\overset{Lemma~\ref{20110516eq2}}= z_{n+1}-x_1^{-c_{n}} x_2^{-c_{n-1}}\sum_{\mathbf{\beta}\in\mathcal{T}^{\geq 3}(\mathcal{D}_{n+1})}x_1^{r|\mathbf{\beta}|_1}x_2^{r(c_{n}-|\mathbf{\beta}|_2)}\\
&\overset{(\ref{20110411z})}= x_1^{-c_{n}} x_2^{-c_{n-1}}\sum_{\mathbf{\beta}\in\widetilde{\mathcal{F}}(\mathcal{D}_{n+1})\setminus\mathcal{T}^{\geq 3}(\mathcal{D}_{n+1})}x_1^{r|\mathbf{\beta}|_1}x_2^{r(c_{n}-|\mathbf{\beta}|_2)}\\
&\overset{(\ref{20110516eq})}=x_1^{-c_{n}} x_2^{-c_{n-1}}\sum_{\mathbf{\beta}\in{\mathcal{F}}(\mathcal{D}_{n+1})}x_1^{r|\mathbf{\beta}|_1}x_2^{r(c_{n}-|\mathbf{\beta}|_2)}.\endaligned$$
\end{proof}

In order to prove Lemma~\ref{20110411lem1}, we need the following notation. 

\begin{defn}
The sequence $\{b_{i,j}\}_{i\in \mathbb{Z}_{\geq 2}, 1\leq j\leq c_i}$ is defined by:
$$b_{i,j}=\left\{\begin{array}{ll} r, &\text{ if }\alpha_j\text{ is a horizontal edge of }\mathcal{D}_{i+1} \\
r-1,&\text{ if }\alpha_j\text{ is a vertical edge of }\mathcal{D}_{i+1}.    \end{array}   \right.$$
\qed\end{defn}

For integers $i\leq j$, we denote the set $\{i,i+1,i+2,\cdots,j\}$ by $[i,j]$. We will always identify $[i,j]$ with the subpath given by $(\alpha_i,\alpha_{i+1},\cdots,\alpha_j)$. 

\begin{defn}\label{def_of_f}
We will need a function $f$ from $\{\text{subsets of } [1,c_{n-1}]\}$ to $\{\text{subsets of } [1,c_{n}]\}$. For each subset $V\subset [1,c_{n-1}]$, we define $f(V)$ as follows.

If $V=\emptyset$ then $f(\emptyset)=\emptyset$. If $V\neq\emptyset$ then we write $V$ as a disjoint union of maximal connected subsets  $ V=\sqcup_{i=1}^j [e_i,e_i+\ell_i-1]$ with $\ell_i>0$ $(1\leq i\leq j)$ and $e_i+\ell_i<e_{i+1}$ $(1\leq i\leq j-1)$. For each $1\leq i\leq j$, let $$W_i=[1+\sum_{k=1}^{e_i-1}b_{n-1,k},\, \sum_{k=1}^{e_i+\ell_i-1}b_{n-1,k}]$$ and define $f_i(V)$ by
$$
f_i(V):=\left\{ \begin{array}{ll}  
W_i,  &\begin{array}{l}\text{if the subpath given by }W_i\text{ is blue or green;} \end{array}\\
 \text{ } & \text{ }\\
\{\sum_{k=1}^{e_i-1}b_{n-1,k}\}\cup W_i, & \text{ otherwise.} \end{array}  \right.
$$
Then $f(V)$ is obtained by taking the union of $f_i(V)$'s:
$$
f(V):=\cup_{i=1}^j f_i(V).
$$Note that the subpath given by $f_i(V)$ is always one of blue, green, or red subpaths, and that every blue, green, or red subpath can be realized as the image of a maximal connected interval under $f$.
\qed\end{defn}

\begin{exmp}
Let $r=3$ and $n=5$. Then $f(\{4,5,6\})=\{9,10,11,12,13,14,15,16\}.$ As illustrated below, the image of the subpath $(\alpha_4,\alpha_5,\alpha_6)$ under $f$ is the subpath $(\alpha_9,\cdots,\alpha_{16}),$ which is blue.

$$
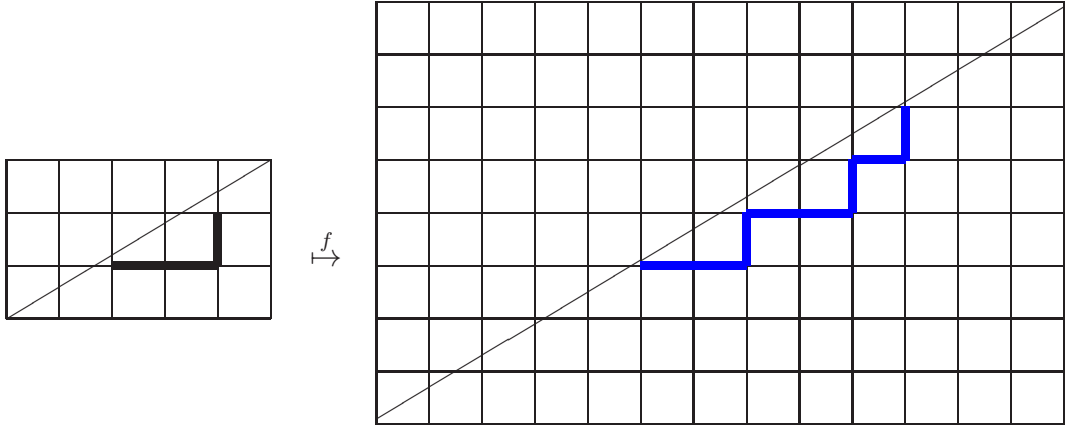
\begin{picture}(140,60)
\boxs{0,60}\boxs{20,60}\boxs{40,60}\boxs{60,60}\boxs{80,60}
\boxs{0,80}\boxs{20,80}\boxs{40,80}\boxs{60,80}\boxs{80,80}
\boxs{0,40}\boxs{20,40}\boxs{40,40}\boxs{60,40}\boxs{80,40}
\linethickness{1pt}\put(0,40){\line(5,3){100}}
\linethickness{3pt}\put(40,60){\line(1,0){40}}
\linethickness{3pt}\put(80,60){\line(0,1){20}}\put(115,60){$\Huge{\overset{f}\mapsto}$}
\end{picture} \begin{picture}(260,140)
\boxs{0,0}\boxs{20,0}\boxs{40,0}\boxs{60,0}\boxs{80,0}
\boxs{100,0}\boxs{120,0}\boxs{140,0}\boxs{160,0}\boxs{180,0}
\boxs{200,0}\boxs{220,0}\boxs{240,0}
\boxs{0,20}\boxs{20,20}\boxs{40,20}\boxs{60,20}\boxs{80,20}
\boxs{100,20}\boxs{120,20}\boxs{140,20}\boxs{160,20}\boxs{180,20}
\boxs{200,20}\boxs{220,20}\boxs{240,20}
\boxs{0,40}\boxs{20,40}\boxs{40,40}\boxs{60,40}\boxs{80,40}
\boxs{100,40}\boxs{120,40}\boxs{140,40}\boxs{160,40}\boxs{180,40}
\boxs{200,40}\boxs{220,40}\boxs{240,40}
\boxs{0,60}\boxs{20,60}\boxs{40,60}\boxs{60,60}\boxs{80,60}
\boxs{100,60}\boxs{120,60}\boxs{140,60}\boxs{160,60}\boxs{180,60}
\boxs{200,60}\boxs{220,60}\boxs{240,60}
\boxs{0,80}\boxs{20,80}\boxs{40,80}\boxs{60,80}\boxs{80,80}
\boxs{100,80}\boxs{120,80}\boxs{140,80}\boxs{160,80}\boxs{180,80}
\boxs{200,80}\boxs{220,80}\boxs{240,80}
\boxs{0,100}\boxs{20,100}\boxs{40,100}\boxs{60,100}\boxs{80,100}
\boxs{100,100}\boxs{120,100}\boxs{140,100}\boxs{160,100}\boxs{180,100}
\boxs{200,100}\boxs{220,100}\boxs{240,100}
\boxs{0,120}\boxs{20,120}\boxs{40,120}\boxs{60,120}\boxs{80,120}
\boxs{100,120}\boxs{120,120}\boxs{140,120}\boxs{160,120}\boxs{180,120}
\boxs{200,120}\boxs{220,120}\boxs{240,120}
\boxs{0,140}\boxs{20,140}\boxs{40,140}\boxs{60,140}\boxs{80,140}
\boxs{100,140}\boxs{120,140}\boxs{140,140}\boxs{160,140}\boxs{180,140}
\boxs{200,140}\boxs{220,140}\boxs{240,140}
\linethickness{3pt}\put(0,2){\line(5,3){260}}
\linethickness{3pt}\color{blue}\put(100,60){\line(1,0){40}}
\linethickness{3pt}\put(140,60){\line(0,1){20}}\put(140,80){\line(1,0){40}}\put(180,80){\line(0,1){20}}\put(180,100){\line(1,0){20}}\put(200,100){\line(0,1){20}}
\end{picture}$$
\end{exmp}

\begin{defn}
Same notation as in Definition~\ref{def_of_f}. For each $1\leq i\leq j$, we define $\delta_{[e_i,e_i+\ell_i-1]}$ by
$$
\left\{\begin{array}{l}1,\,\,\text{ if the }w\text{-th edge in the subpath corresponding }\\
\,\,\,\,\,\,\,\,\,\,\text{ to }[e_i,e_i+\ell_i-1]\text{ is vertical for some }2\leq w\leq r-1;\\
0,\,\,\text{ otherwise}.\end{array}\right.
$$
Let $\delta_V:=\sum_{i=1}^j\delta_{[e_i,e_i+\ell_i-1]}$.
\qed\end{defn}

\begin{defn}[pull-back]
Let $n>1$ be any integer. For any subset $W\subset [1,c_n]$, the pull-back $f^*(W)$ by $f$ is defined as follows:
$$f^*(W)=\cup_{V : f(V)\subset W} V\subset [1,c_{n-1}].$$
\qed\end{defn}

The next lemma will be needed to prove Lemma~\ref{20110411lem1}. 

\begin{lem}\label{thenumofelem}
Let $V$ be any subset of $[1,c_{n-1}]$. Then
$$|f(V)|=r|V|-|f^*(V)|-\delta_V.$$
\end{lem}
\begin{proof}
The question is local, so we may assume that $V$ is an interval.
First, assume that $V$ is an interval such that the $(r-w)$-th edge in the corresponding subpath, denoted also by $V$, is not vertical for any $1\leq w\leq r-2$. Observe that $|f(V)|$ is determined by how many horizontal edges and vertical edges are in $V$ and whether the subpath given by $f(V)$ is blue/green or red. Suppose that the subpath given by $f(V)$ is blue or green. Then $|f(V)|=r|V|-($the number of vertical edges in $V)$. Also if the subpath given by $f(V)$ is blue or green, then (the number of vertical edges in $V)=|f^*(V)|$, because the $(w+1)-$th edge in $V$ is not vertical for any $1\leq w\leq r-2$. 

Suppose that the subpath given by $f(V)$ is red. Then $|f(V)|=r|V|-($the number of vertical edges in $V)+1$. It is not hard to show that (the number of vertical edges in $V)=|f^*(V)|+1$.

If $V$ is  an interval such that the $(r-w)$-th edge is vertical for some $1\leq w\leq r-2$, then $f(V)$ is green. Then $|f(V)|=r|V|-($the number of vertical edges in $V)$, and (the number of vertical edges in $V)=|f^*(V)|+1$. 
\end{proof}

The binomial coefficients we will use are generalized binomial coefficients, i.e., for any (possibly negative) integers $A,B$, 
$${A\choose B} := \left\{ \begin{array}{ll}  \frac{\prod_{i=0}^{B-1} (A-i)}{B!}, & \text{ if } B>0\\ \, & \,  \\   1, & \text{ if }B=0 \\ \, & \, \\  0, & \text{ if }B<0.  \end{array}  \right.$$

\begin{proof}[Proof of Lemma~\ref{20110411lem1}]
$$\aligned
F(z_{n})&=x_2^{-c_{n-1}} \sum_{\mathbf{\beta}\in\widetilde{\mathcal{F}}(\mathcal{D}_n)} x_2^{r|\mathbf{\beta}|_1}\left(\frac{x_2^r+1}{x_1}\right)^{r(c_{n-1}-|\mathbf{\beta}|_2)-c_{n-2}}\\
&=x_1^{-c_n}x_2^{-c_{n-1}}\sum_{\mathbf{\beta}\in\widetilde{\mathcal{F}}(\mathcal{D}_n)}x_1^{r|\mathbf{\beta}|_2}x_2^{r|\mathbf{\beta}|_1}\left(x_2^r+1\right)^{c_n-r|\mathbf{\beta}|_2}
\\
&=x_1^{-c_n}x_2^{-c_{n-1}}\sum_{\mathbf{\beta}\in\widetilde{\mathcal{F}}(\mathcal{D}_n)} \sum_{\theta\in\mathbb{Z}}{{c_{n}-r|\mathbf{\beta}|_2}\choose {\theta}}x_1^{r|\mathbf{\beta}|_2}x_2^{r|\mathbf{\beta}|_1}x_2^{r\theta}\\
&=x_1^{-c_n}x_2^{-c_{n-1}} \sum_{V\subset[1,c_{n-1}]}\,\,\sum_{\beta:\cup\beta_i=V}\sum_{\theta\in\mathbb{Z}}{{c_{n}-r|V|}\choose {\theta}}x_1^{r|V|}x_2^{r|\mathbf{\beta}|_1}x_2^{r\theta}.\,\,\,\cdots\cdots\cdots(\dag\dag)\endaligned$$

For a given $V\subset [1,c_{n-1}]$, consider any subset, say $U$, of $f^*(V)$. Then the image of each  maximal connected subset of $U$  under $f$ is a blue, green, or red subpath. So $U$ uniquely determines $\beta$, namely,
$$\beta=\{f(\text{all maximal connected subsets of }U)\} \cup \{\text{all individual edges in }V\setminus f(U)\}.
$$Conversely,  $\beta$ with $\cup\beta_i=V$ uniquely determines a subset $U$ of $f^*(V)$, that is, $$U=f^{-1}(\text{all blue, green, or red subpaths in }\beta),$$
because $f^{-1}(\text{a blue, green, or red subpath})$ is well-defined and no subpath of $V\setminus f(f^*(V))$ can be a blue, green, or red subpath.

Since $|\beta|_1=|U|$, we can see that $(\dag\dag)$ becomes
$$\aligned
F(z_{n})&=x_1^{-c_n}x_2^{-c_{n-1}} \sum_{V\subset[1,c_{n-1}]}\,\,\sum_{U\subset f^*(V)}\sum_{\theta\in\mathbb{Z}}{{c_{n}-r|V|}\choose {\theta}}x_1^{r|V|}x_2^{r|U|}x_2^{r\theta}\\
&=x_1^{-c_n}x_2^{-c_{n-1}} \sum_{V\subset[1,c_{n-1}]}\,\,\sum_{\rho\in\mathbb{Z}}\,\,\sum_{\theta\in\mathbb{Z}}{{c_{n}-r|V|}\choose {\theta}}{{|f^*(V)|}\choose \rho}x_1^{r|V|}x_2^{r(\theta+\rho)}\\
&\overset{(**)}=x_1^{-c_n}x_2^{-c_{n-1}} \sum_{V\subset[1,c_{n-1}]}\,\,\sum_{\gamma\in\mathbb{Z}}{{c_{n}-r|V|+|f^*(V)|}\choose \gamma}x_1^{r|V|}x_2^{r\gamma}\\
&=x_1^{-c_n}x_2^{-c_{n-1}} \sum_{V\subset[1,c_{n-1}]}\,\,\sum_{\gamma\in\mathbb{Z}}{{c_{n}-r|V|+|f^*(V)|}\choose {c_{n}-r|V|+|f^*(V)|-\gamma}}x_1^{r|V|}x_2^{r(c_{n}-r|V|+|f^*(V)|-\gamma)}\\
&\overset{Lemma~\ref{thenumofelem}}=x_1^{-c_n}x_2^{-c_{n-1}} \sum_{V\subset[1,c_{n-1}]}\,\,\sum_{\gamma\in\mathbb{Z}}{{c_{n}-|f(V)|-\delta_V}\choose {c_{n}-|f(V)|-\delta_V-\gamma}}x_1^{r|V|}x_2^{r(c_{n}-|f(V)|-\delta_V-\gamma)}\\
&=x_1^{-c_n}x_2^{-c_{n-1}} \sum_{V\subset[1,c_{n-1}]}\,\,\sum_{h\in\mathbb{Z}}{{c_{n}-|f(V)|-\delta_V}\choose {c_{n}-|f(V)|-h}}x_1^{r|V|}x_2^{r(c_{n}-|f(V)|-h)}\\
&=x_1^{-c_n}x_2^{-c_{n-1}} \sum_{V\subset[1,c_{n-1}]}\,\,\sum_{h\in\mathbb{Z}}{{c_{n}-|f(V)|-\delta_V}\choose {h-\delta_V}}x_1^{r|V|}x_2^{r(c_{n}-|f(V)|-h)},\,\,\,\,\cdots\cdots\cdots(\ddag)
\endaligned$$
where $(**)$ follows from the Chu--Vandermonde identity. 

Next we analyze $z_{n+1}$. If $\mathbf{\beta}\in\widetilde{\mathcal{F}}(\mathcal{D}_{n+1})$, then there exist 
\begin{equation}\label{05192011eq}\left\{\aligned
&e\in \mathbb{Z}_{\geq 0},\\
&0\leq i_1<k_1<i_2<k_2<\cdots<i_e<k_e\leq c_{n-1},\\
&h \in \mathbb{Z}_{\geq 0},\text{ and }\\
&{j_1}<\cdots<{j_h}\in [1,c_n]\setminus  f(\sqcup_{\ell=1}^e [1+i_\ell,k_\ell])
\endaligned\right.\end{equation}such that $$\beta=\{\alpha(i_1,k_1),\cdots,\alpha(i_e,k_e)\}\cup\{\alpha_{j_1},\cdots, \alpha_{j_h}\}.$$ For a given $\mathbf{\beta}\in\widetilde{\mathcal{F}}(\mathcal{D}_{n+1})$, this expression is unique.
Conversely $(\ref{05192011eq})$ uniquely determines an element $\beta$ in $\widetilde{\mathcal{F}}(\mathcal{D}_{n+1})$. Note that $|\mathbf{\beta}|_1=\sum_{\ell=1}^e (k_\ell-i_\ell)$ and $|\mathbf{\beta}|_2=|f(\sqcup_{\ell=1}^e [1+i_\ell,k_\ell])|+h$.

So we have
$$\aligned z_{n+1}&=x_1^{-c_{n}} x_2^{-c_{n-1}}\sum_{\mathbf{\beta}\in\widetilde{\mathcal{F}}(\mathcal{D}_{n+1})}x_1^{r|\mathbf{\beta}|_1}x_2^{r(c_{n}-|\mathbf{\beta}|_2)}\\
&=x_1^{-c_{n}} x_2^{-c_{n-1}}\sum_{e\in \mathbb{Z}_{\geq 0}}\,\, \sum_{1\leq i_1<\cdots<k_e\leq c_{n-1}}\,\, \sum_{h \in \mathbb{Z}_{\geq 0}}\,\,\sum_{{j_1}<\cdots<{j_h}\in [1,c_n]\setminus  f(\sqcup_\ell [1+i_\ell,k_\ell])}
x_1^{r\sum_{l} (k_\ell-i_\ell)}x_2^{r(c_{n}-|f(\sqcup_\ell [1+i_\ell,k_\ell])|-h)}\\
&=x_1^{-c_n}x_2^{-c_{n-1}} \sum_{V\subset[1,c_{n-1}]}\,\,\sum_{h \in \mathbb{Z}_{\geq 0}}\,\,\sum_{{j_1}<\cdots<{j_h}\in [1,c_n]\setminus  f(V)}
x_1^{r|V|}x_2^{r(c_{n}-|f(V)|-h)}\\
&=x_1^{-c_n}x_2^{-c_{n-1}} \sum_{V\subset[1,c_{n-1}]}\,\,\sum_{h \in \mathbb{Z}}\,\,{{c_{n}-|f(V)|}\choose {h}} x_1^{r|V|}x_2^{r(c_{n}-|f(V)|-h)}\\
&\overset{\text{by }(\ddag)}=F(z_{n})+x_1^{-c_n}x_2^{-c_{n-1}} \sum_{V\subset[1,c_{n-1}]}\,\,\sum_{h \in \mathbb{Z}}\,\,\left({{c_{n}-|f(V)|}\choose {h}}-{{c_{n}-|f(V)|-\delta_V}\choose {h-\delta_V}}\right) x_1^{r|V|}x_2^{r(c_{n}-|f(V)|-h)}\\
&=F(z_{n})+x_1^{-c_{n}} x_2^{-c_{n-1}}\sum_{\mathbf{\beta}\in\mathcal{T}^{\geq 3}(\mathcal{D}_{n+1})\setminus \mathcal{T}^{\geq 4}(\mathcal{D}_{n+1})}x_1^{r|\mathbf{\beta}|_1}x_2^{r(c_{n}-|\mathbf{\beta}|_2)},
\endaligned$$where the last equality is a consequence of the next Lemma. 
\end{proof}

\begin{lem}\label{05212011lem}
Let $V$ be any subset of $[1,c_{n-1}]$ and $h$ be any integer. Then the number of elements $\beta$ in $\widetilde{\mathcal{F}}(\mathcal{D}_{n+1})$ satisfying \begin{equation}\label{05212011eq02}\left\{\aligned &f(V)\text{ is the union of all blue, green, and red subpaths in }\beta,\,\\
&|\mathbf{\beta}|_2=|f(V)|+h,\text{ and }\\
&\mathbf{\beta}\in\mathcal{T}^{\geq 3}(\mathcal{D}_{n+1})\setminus \mathcal{T}^{\geq 4}(\mathcal{D}_{n+1})\endaligned\right.\end{equation} is equal to
\begin{equation}\label{05212011eq03}{{c_{n}-|f(V)|}\choose {h}}-{{c_{n}-|f(V)|-\delta_V}\choose {h-\delta_V}}.\end{equation} 
\end{lem}
\begin{proof}
We remember the definition of $(m,w)$-green subpaths  in Definition~\ref{alpha(i,k)} (2-a), especially for $m=3$. For an interval $[i,k]\subset [1,c_{n-1}]$, the subpath in $\mathcal{D}_n$ corresponding to $[i+1,i+r-1]\cap[i,k]$ contains a vertical edge if and only if the slope in $\mathcal{D}_{n+1}$ between the point $v_i$ and the point corresponding to the upper endpoint of the vertical edge is as large as possible, more precisely,
\begin{equation}\label{05212011eq01}\text{ the subpath in }\mathcal{D}_{n+1}\text{ corresponding to }f([i,k])\text{ is }(3,w)\text{-green for some }w\in[1,r-2].\end{equation}

If no maximal connected interval of $V$ satisfies $(\ref{05212011eq01})$, then $\delta_V=0$ and there is no $\beta$ satisfying (\ref{05212011eq02}), so the statement holds true. If only one maximal connected interval, say $[i,k]$, of $V$ satisfies $(\ref{05212011eq01})$, then $\delta_V=1$. On the other hand, since $$\mathbf{\beta}\in\mathcal{T}^{\geq 3}(\mathcal{D}_{n+1})\setminus \mathcal{T}^{\geq 4}(\mathcal{D}_{n+1}),$$ none of the $(c_{2}-wc_{1})$ preceding edge(s) of $v_i$ is contained in any element $\beta_{j'}$ of $\beta$. As $c_{2}-wc_{1}=1$, the number of $\beta$ satisfying (\ref{05212011eq02}) is obtained by subtracting the number of sequences ${j_1}<\cdots<{j_h}\in [1,c_n]\setminus  f(V)$ with $\left(\min f([i,k])-1\right)\in\{j_1,\cdots,j_h\}$ from the number of sequences ${j_1}<\cdots<{j_h}\in [1,c_n]\setminus  f(V)$, which is ${{c_{n}-|f(V)|}\choose {h}}-{{c_{n}-|f(V)|-1}\choose {h-1}}$.

Similarly one can verify the statement in the case that more than one maximal connected intervals of $V$ satisfy $(\ref{05212011eq01})$.
\end{proof}

It remains to prove Lemma~\ref{20110411lem2}.

\begin{proof}[Sketch of Proof of Lemma~\ref{20110411lem2}]
Here we will deal only with the case of $n=u+2$. The case of $n>u+2$ makes use of the same argument. As we use the induction (\ref{globalinduction}), we can assume that $$x_i=x_1^{-c_{i-1}} x_2^{-c_{i-2}}\sum_{\mathbf{\beta}\in\mathcal{F}(\mathcal{D}_i)}x_1^{r|\mathbf{\beta}|_1}x_2^{r(c_{i-1}-|\mathbf{\beta}|_2)}$$for $i\leq n$. 

For any $w\in [1,r-2]$, it is easy to show that the lattice point
$(w(c_{n-2}-c_{n-3}), wc_{n-3})$ is below the diagonal from $(0,0)$ to $(c_{n-1}-c_{n-2},c_{n-2})$ and that the points $(w(c_{n-2}-c_{n-3}), 1+wc_{n-3})$ and $(w(c_{n-2}-c_{n-3})-1, wc_{n-3})$ are above the diagonal. So $(w(c_{n-2}-c_{n-3}), wc_{n-3})$ is one of the vertices $v_i$ on $\mathcal{D}_n$. Actually $v_{wc_{n-3}}=(w(c_{n-2}-c_{n-3}), wc_{n-3})$. Since $u=n-2$ and $\alpha(wc_{n-3},c_{n-2})$ is the only $(n-2,w)$-green subpath in $\{\alpha(i,k)\,|\, 0 \leq i < k \leq c_{n-2}\}$, every $\mathbf{\beta}\in\mathcal{T}^{\geq u}(\mathcal{D}_n)\setminus \mathcal{T}^{\geq u+1}(\mathcal{D}_n)$ must contain the green subpath from $v_{wc_{n-3}}$ to $v_{c_{n-2}}$. Then none of the 
$c_{n-3}-wc_{n-4}$ preceding edges of $v_{wc_{n-3}}$ is contained in any element $\beta_{j'}$ of $\beta$. The green subpath from $v_{wc_{n-3}}$ to $v_{c_{n-2}}$ corresponds to the interval $[wc_{n-2}+1,c_{n-1}] \subset [1,c_{n-1}]$. The 
$c_{n-3}-wc_{n-4}$ preceding edges of $v_{wc_{n-3}}$ are $\alpha_{(rw-1)c_{n-3}+1},\cdots, \alpha_{wc_{n-2}}$.

Thus we have
$$\aligned 
&x_1^{-c_{n-1}} x_2^{-c_{n-2}}\sum_{\mathbf{\beta}\in\mathcal{T}^{\geq u}(\mathcal{D}_n)\setminus \mathcal{T}^{\geq u+1}(\mathcal{D}_n)}x_1^{r|\mathbf{\beta}|_1}x_2^{r(c_{n-1}-|\mathbf{\beta}|_2)}\\
&=\sum_{w=1}^{r-2} x_1^{-c_{n-1}} x_2^{-c_{n-2}}\sum_{V\subset [1,(rw-1)c_{n-3}]}\,\,\sum_{\mathbf{\beta}: \cup\beta_i=V\cup[wc_{n-2}+1,c_{n-1}],\,\mathbf{\beta}\ni \alpha(wc_{n-3},c_{n-2})}x_1^{r|\mathbf{\beta}|_1}x_2^{r(c_{n-1}-|\mathbf{\beta}|_2)}.\,\,\,\,\,\,\,\,(*)\endaligned$$

We observe that the subpath corresponding to $[1,(rw-1)c_{n-3}]$ consists of $(w-1)$ copies of $\mathcal{D}_{n-1}$, $(r-1)$ copies of $\mathcal{D}_{n-2}$, and $(w-1)$ copies of $\mathcal{D}_{n-3}$. Let $v_{j_0}=(0,0)$ and $v_{j_i}$ be the end point of each of these copies, i.e., $$\aligned &v_{j_i}=v_{ic_{n-3}}\text{ for }1\leq i\leq w-1,\\
& v_{j_{w-1+i}}=v_{(w-1)c_{n-3}+ic_{n-4}}\text{ for }1\leq i\leq r-1,\\
& v_{j_{w+r-2+i}}=v_{(w-1)c_{n-3}+(r-1)c_{n-4}+ic_{n-5}}\text{ for }1\leq i\leq w-1.\endaligned$$
If a $(m,w')$-green (resp. blue or red) subpath, say $\alpha(i,k)$, in $[1,(rw-1)c_{n-3}]$ passes through $v_{j_e},v_{j_{e+1}},\cdots, v_{j_{e+\ell}}$, then $\alpha(i,k)$ can be naturally decomposed into $\alpha(i,j_e)$, $\alpha(j_e,j_{e+1})$, $\cdots$,  $\alpha(j_{e+\ell},k)$. It is not hard to show that $\alpha(i,j_e)$ is also $(m,w')$-green (resp. blue or red) and that 
 $\alpha(j_e,j_{e+1})$, $\cdots$,  $\alpha(j_{e+\ell},k)$ are all blue.

Hence 
$$\aligned
&(*)=\sum_{w=1}^{r-2} x_1^{-c_{n-1}} x_2^{-c_{n-2}}\left(\sum_{\mathbf{\beta}\in\mathcal{F}(\mathcal{D}_{n-1})}x_1^{r|\mathbf{\beta}|_1}x_2^{r(c_{n-2}-|\mathbf{\beta}|_2)}\right)^{w-1}\left(\sum_{\mathbf{\beta}\in\mathcal{F}(\mathcal{D}_{n-2})}x_1^{r|\mathbf{\beta}|_1}x_2^{r(c_{n-3}-|\mathbf{\beta}|_2)}\right)^{r-1}\\
&\,\,\,\,\,\,\,\,\,\,\,\,\,\,\,\,\times\left(\sum_{\mathbf{\beta}\in\mathcal{F}(\mathcal{D}_{n-3})}x_1^{r|\mathbf{\beta}|_1}x_2^{r(c_{n-4}-|\mathbf{\beta}|_2)}\right)^{w-1}x_1^{r(c_{n-2}-wc_{n-3})}x_2^{r(c_{n-3}-wc_{n-4})}\\
&=\sum_{w=1}^{r-2} x_1^{-c_{n-1}} x_2^{-c_{n-2}}\left(x_{n-1}x_1^{c_{n-2}}x_2^{c_{n-3}}\right)^{w-1}\left(x_{n-2}x_1^{c_{n-3}}x_2^{c_{n-4}}\right)^{r-1}\\
&\,\,\,\,\,\,\,\,\,\,\,\,\,\,\,\,\times\left(x_{n-3}x_1^{c_{n-4}}x_2^{c_{n-5}}\right)^{w-1}x_1^{r(c_{n-2}-wc_{n-3})}x_2^{r(c_{n-3}-wc_{n-4})}\\
&=\sum_{w=1}^{r-2} (x_{n-1})^{w-1}(x_{n-2})^{r-1}(x_{n-3})^{w-1}.
\endaligned$$For the same reason, we get $$x_1^{-c_{n}} x_2^{-c_{n-1}}\sum_{\mathbf{\beta}\in\mathcal{T}^{\geq u+1}(\mathcal{D}_{n+1})\setminus \mathcal{T}^{\geq u+2}(\mathcal{D}_{n+1})}x_1^{r|\mathbf{\beta}|_1}x_2^{r(c_{n}-|\mathbf{\beta}|_2)}=\sum_{w=1}^{r-2} (x_{n})^{w-1}(x_{n-1})^{r-1}(x_{n-2})^{w-1}.$$
Therefore, we have
$$\aligned 
&F\left(x_1^{-c_{n-1}} x_2^{-c_{n-2}}\sum_{\mathbf{\beta}\in\mathcal{T}^{\geq u}(\mathcal{D}_n)\setminus \mathcal{T}^{\geq u+1}(\mathcal{D}_n)}x_1^{r|\mathbf{\beta}|_1}x_2^{r(c_{n-1}-|\mathbf{\beta}|_2)}\right)=F\left(\sum_{w=1}^{r-2} (x_{n-1})^{w-1}(x_{n-2})^{r-1}(x_{n-3})^{w-1}\right)\\
&=\sum_{w=1}^{r-2} (x_{n})^{w-1}(x_{n-1})^{r-1}(x_{n-2})^{w-1}=x_1^{-c_{n}} x_2^{-c_{n-1}}\sum_{\mathbf{\beta}\in\mathcal{T}^{\geq u+1}(\mathcal{D}_{n+1})\setminus \mathcal{T}^{\geq u+2}(\mathcal{D}_{n+1})}x_1^{r|\mathbf{\beta}|_1}x_2^{r(c_{n}-|\mathbf{\beta}|_2)}.\endaligned$$
 \end{proof}


\begin{thebibliography}{99}
\bibitem{ADSS} I. Assem, G. Dupont, R. Schiffler and D. Smith,  Friezes, strings and cluster variables, arXiv:1009.3341.
\bibitem{ARS} I. Assem, C. Reutenauer and D. Smith, Friezes.
Adv. Math. {\bf 225} (2010), no. 6, 3134--3165. 
\bibitem{BZ}
A. Berenstein, and A. Zelevinsky, Quantum cluster algebras. Adv. Math. \textbf{195} (2005), no. 2, 405--455. MR2146350 (2006a:20092).

\bibitem{CC}
P. Caldero and F. Chapoton, Cluster algebras as Hall algebras of quiver representations, Comment. Math. Helv. \textbf{81} (2006), No. 3, 595--616.

\bibitem{CK1}
P. Caldero and B. Keller, From triangulated categories to cluster algebras. Invent. Math. \textbf{172} (2008), No. 1, 169--211.




\bibitem{CZ}
P. Caldero and  A. Zelevinsky,  Laurent expansions in cluster algebras via quiver representations,	 Mosc. Math. J. \textbf{6} (2006),  No. 3, 411--429. MR2274858 (2008j:16045).

\bibitem{DWZ}
H. Derksen, J. Weyman and A. Zelevinsky, Quivers with potentials and their representations II: applications to cluster algebras. J. Amer. Math. Soc. \textbf{23} (2010), No. 3, 749--790. MR2629987.

\bibitem{DK}
P. Di Francesco and R. Kedem, Discrete non-commutative integrability: Proof of a conjecture by M. Kontsevich, Intern. Math. Res. Notes, (2010) doi:10.1093/imrn/rnq024.

\bibitem{D}
G. Dupont,  Positivity in coefficient-free rank two cluster algebras, Electron. J. Combin. \textbf{16}  (2009), No. 1. MR2529807. 


\bibitem{FZ}
S. Fomin and  A. Zelevinsky, Cluster algebras I: Foundations, J. Amer. Math. Soc.  \textbf{15} (2002) No. 2  497--529, 2002. MR1887642 (2003f:16050).

\bibitem{FZ4}
S. Fomin and  A. Zelevinsky, Cluster algebras IV: Coefficients, Comp. Math. \textbf{143} (2007), 112--164.

\bibitem{FK} C. Fu and B. Keller, On cluster algebras with coefficients and 2-Calabi-Yau categories.
Trans. Amer. Math. Soc. {\bf 362} (2010), no. 2, 859--895. 

\bibitem{M2}
D.~Grayson, M.~Stillman, Macaulay2, a software system for research
in algebraic geometry, available at http://www.math.uiuc.edu/Macaulay2/.

\bibitem{HL}
D. Hernandez and B. Leclerc, Cluster algebras and quantum affine algebras, 
Duke Math. J. \textbf{154} (2010), No. 2, 265--341.

\bibitem{L} K. Lee, On cluster variables of rank two acyclic cluster algebras. arXiv:1008.1829.

\bibitem{MP}
G. Musiker and J. Propp, Combinatorial interpretations for rank-two cluster algebras of affine type, Electron. J. Combin. \textbf{14} (2006). MR2285819 (2008j:05374).

\bibitem{MSW}
G. Musiker, R. Schiffler and L. Williams, Positivity for cluster algebras from surfaces, to appear in Adv.  Math. doi:10.1016/j.aim.2011.04.018.


\bibitem{N}
H. Nakajima, Quiver varieties and cluster algebras, Kyoto J. Math. \textbf{51} (2011), No. 1, 71--126. (Memorial Issue for the Late Professor Masayoshi Nagata).

\bibitem{Q}
F. Qin, Quantum cluster variables via Serre polynomials, arXiv:1004.4171, to appear in Journal f\"ur die reine und angewandte Mathematik (Crelle's Journal).

\bibitem{S3} R. Schiffler, On cluster algebras arising from unpunctured surfaces. II. Adv. Math. 223 (2010), no. 6, 1885--1923.
\bibitem{S2} R. Schiffler, A cluster expansion formula (An case). Electron. J. Combin. 15 (2008), no. 1, Research paper 64, 9 pp.
\bibitem{ST}  R. Schiffler  and H. Thomas, On cluster algebras arising from unpunctured surfaces. Int. Math. Res. Not. IMRN 2009, no. 17, 3160--3189.
\bibitem{SZ}
P. Sherman and A. Zelevinsky, Positivity and canonical bases in rank 2 cluster algebras of finite and affine types,  Mosc. Math. J. \textbf{4} (2004) No. 4, 947--974. MR2124174 (2006c:16052).



 \end{thebibliography}
\end{document}